\documentclass[a4paper,11pt,reqno]{article}

\usepackage{relsize}
\usepackage{cite}
\usepackage{color}
\usepackage{hyperref}
\hypersetup{
  colorlinks   = true, 
  urlcolor     = green, 
  linkcolor    = blue, 
  citecolor   = red 
}
\usepackage{amsmath,amsthm,amssymb}

\usepackage[margin=2.6cm]{geometry}

\makeatletter
\usepackage{comment}
\let\wfs@comment@comment\comment
\let\comment\@undefined

\usepackage{changes}
\let\wfs@changes@comment\comment
\let\comment\@undefined

\newcommand\comment{%
    \ifthenelse{\equal{\@currenvir}{comment}}
    {\wfs@comment@comment}
    {\wfs@changes@comment}%
}

\definechangesauthor[name=Daniele, color=red]{DAN}
\definechangesauthor[name=Martino, color=blue]{MAR}

\usepackage{stmaryrd}

\theoremstyle{definition}
\newtheorem{theorem}{Theorem}[section]
\newtheorem{definition}[theorem]{{{Definition}}}
\newtheorem{example}[theorem]{{{Example}}}

\newtheorem{remark}[theorem]{{{Remark}}}

\newtheorem{corollary}[theorem]{{{Corollary}}}

\newtheorem{lemma}[theorem]{{{Lemma}}}

\newcommand{\numberset}{\mathbb}

\newcommand{\F}{\numberset{F}}

\newcommand{\HHH}{\mathcal{H}}

\newcommand{\C}{\mathcal{C}}

\newcommand{\mC}{\mathcal{C}}

\newcommand{\sH}{\sigma}

\newcommand{\wt}{\textnormal{wt}}

\DeclareMathOperator{\PG}{PG}

\newcommand{\wH}{\textnormal{wt}}

\title{\textbf{Small Strong Blocking Sets by Concatenation}}

\usepackage{authblk}
\author[1]{Daniele Bartoli}
\affil[1]{Universit\`a degli Studi di Perugia, Italy}

\author[2]{Martino Borello}
\affil[2]{Universit\'e Paris 8, Laboratoire de G\'eom\'etrie, Analyse et Applications, LAGA,
Universit\'e Sorbonne Paris Nord, CNRS, UMR 7539, France}

\date{}

\usepackage{setspace}

\usepackage{enumitem}
\setitemize{itemsep=0em}
\setenumerate{itemsep=0em}

\begin{document}

\maketitle


\begin{abstract}
Strong blocking sets and their counterparts, minimal codes,  attracted lots of attention in the last years. Combining  the  concatenating construction of codes with a geometric insight into the minimality condition, we explicitly provide infinite families of small strong blocking sets, whose size is linear in the dimension of the ambient projective spaces. As a byproduct, small saturating sets are obtained.
\end{abstract}

\section*{Introduction}

A set of points $\mathcal{B}$ in a projective space is called a \emph{blocking set} if any hyperplane  contains at least one point of $\mathcal{B}$. A blocking set $\mathcal{B}$ is \emph{strong} if any hyperplane section of $\mathcal{B}$ generates the hyperplane itself; see \cite{1930-5346_2011_1_119,zbMATH06340223,bonini2020minimal}. Although blocking sets were deeply investigated in the last decades (see for example \cite{bs} and references therein), much less is known about strong blocking sets. Strong blocking sets are the geometrical counterpart of an important class of error correcting codes, the so-called \emph{minimal codes}; see \cite{alfarano2019geometric,Full_Characterization}.  One of the main open questions concerning strong blocking sets (in relation to the parameters of the associated minimal code) concerns the \emph{explicit} constructions of \emph{infinite} families of strong blocking sets whose size grows \emph{linearly} with the dimension of the ambient projective space. 

In this paper we exploit  a machinery introduced in  \cite{Cohen_2013} to provide  families of asymptotically good minimal codes  via codes concatenation. The main result of our paper is what we call \emph{outer AB condition} (in analogy to the Ashikhmin-Barg condition proved in \cite{ashikhmin1998minimal}): if $\C$ is a linear code over $\F_{q^k}$ such that
\[\max_{c\in \C} \wt(c)<\frac{q}{q-1} \min_{c\in \C-\{0\}} \wt(c)\]
then the concatenation of $\C$ with a minimal code is minimal. So the corresponding set of points is a strong blocking set. We show that some (existing) MDS codes have such property, making effective the construction of minimal codes that satisfy the outer AB condition. 
We obtain explicit constructions of small strong blocking sets in projective spaces, whose size linearly depends on the dimension of the ambient space. In the last part, we produce some examples showing that in some cases concatenation provides the smallest known strong blocking sets. Finally, as a byproduct, we obtain explicit constructions of saturating sets for specific dimensions whose size is the smallest one can find in the literature. 


\paragraph{Outline.} \  The paper is organized into four sections. Section~\ref{sec:background} contains the preliminaries on minimal codes, strong blocking sets and the known bounds on their size. In Section~\ref{sec:strong} we present a geometric interpretation of the well-known Ashikhmin-Barg condition, we introduce the outer AB condition and we present some explicit construction of small strong blocking sets. Section~\ref{sec:asymptotic} is devoted to some asymptotic results, both theoretical and constructive. Finally, we give some numerical results and we apply some of the previous results to the study of $\rho$-saturating sets in Section~\ref{sec:application}.

\section{Background}\label{sec:background}

\subsection{Minimal codes}

For a vector $v \in \F_q^n$, the \emph{Hamming weight} of $v$ is $\wH(v):=|\sH(v)|$, where $\sH(v):=\{i\in \{1,\ldots,n\} \mid v_i \neq 0\}$ is the Hamming support of $v$. An $[n,k,d]_q$ \emph{linear code} $\C$ is a subspace of $\F_q^n$ of \emph{dimension} $k$ with  \emph{minimum weight}
\[d=d(\C):=\min\{\wH(c) \mid c \in \mC, \, c \neq 0\}\]
and its elements are called \emph{codewords}. Normally, these are the fundamental parameters for error correcting purposes, but in our context the \emph{maximum weight} $w(\C):=\max\{\wH(c)\mid c\in \mC\}$ will play a crucial role. A \emph{generator matrix} $G$ of $\mC$ is a matrix whose rows form a basis for $\mC$.
A linear code $\mC$ is called \emph{projective} (resp. \emph{non-degenerate}) if in one (and thus in all) generator matrix $G$ of $\mC$ no two columns are proportional (resp. no column is the zero vector). 
Two linear codes are said to be (monomially) \emph{equivalent} if there is a monomial transformation sending one into the other. A family of codes is said to be \emph{asymptotically good} if it admits an infinite subfamily of increasing length with both dimension and minimum weight growing linearly in the length. 

\begin{definition} \label{def:mc}
A nonzero codeword $c$ of a code $\mC$ over $\F_q$ is called \emph{minimal} 
if every nonzero codeword $c^{\prime}$ in $\mC$ with $\sigma(c^{\prime}) \subseteq \sigma(c)$ is $\lambda c$ for some nonzero $\lambda\in \F_q$. A linear code $\mC$ is a \emph{minimal code} if all its codewords are minimal.
\end{definition}

\begin{example}\label{example:simplexcode}
A \emph{simplex code} $\mathcal{S}_q(k)$ of dimension $k$ over $\F_q$ is a code defined up to equivalence as follows: its generator matrix is obtained by choosing as columns a nonzero vector from each 1-dimensional subspace of $\F_q^k$. It is easy to prove that its parameters are $[(q^k-1)/(q-1),k,q^{k-1}]_q$ and that all nonzero codewords have the same weight. This last property implies that $\mathcal{S}_q(k)$ is minimal.
\end{example}

The importance of minimal codes relies on their connection with cryptography and coding theory. In fact, minimal codewords in linear codes  were used by Massey \cite{Massey1993,Massey1995} to determine the access structure in a code-based secret sharing scheme.  Previously,  minimal codewords were  studied in connection with decoding algorithms \cite{MR551274}. In the last years, minimal codes attracted  attention and many of the constructions are connected with  few-weight codes (see e.g. \cite{SHI202160,8915756,SHI2020111840}) or exploit the so-called Ashikhmin-Barg condition  \cite{ashikhmin1998minimal}. In the last few years, geometric methods have played a crucial role, as we will see in the following sections. Finally, in a very recent work \cite{alfaranolinear}, minimal rank-metric codes are investigated, in connection with linear sets. As a byproduct, new constructions and perspectives on minimal codes are obtained. 

In this paper, the \emph{concatenation of codes}, a standard way of combining codes to obtain a code of
larger length, plays a crucial role. To describe this process, let $n,k$ be two integers such that $n\geq k$ and let $\pi:\F_{q^k}\to \F_q^n$ be an $\F_q$-linear injection. In the concatenation process, the $\F_{q^k}$-linear code $\mathcal{C}$  is called
\emph{outer code} and the $\F_q$-linear code $\mathcal{I}:=\pi(\F_{q^k})$ is
called \emph{inner code}. If $\C$ is an $\F_{q^k}$-linear code of length $N$ and $\F_{q^k}$-dimension $K$, then the $\F_q$-linear code
$$\mathcal{I} \ \Box_\pi \ \mathcal{C}:=\{(\pi(c_1),\ldots,\pi(c_N))\ | \ (c_1,\ldots,c_N)\in \C\}\subseteq \F_q^{N\cdot n}$$
is called the concatenation of $\mathcal{C}$ with
$\mathcal{I}$ by $\pi$, or simply the \emph{concatenated
code}. This last depends on the choice of $\pi$, but
 some of its properties are independent of  $\pi$. For example,
 $\mathcal{I} \ \Box_\pi \ \mathcal{C}$ is an $[N\cdot n,K\cdot k,\geq d(\mathcal{I})\cdot
d(\mathcal{C})]_q$ code for every choice of $\pi$ (see for example \cite[\S 1.2.3.]{tsfasman1991algebraic}). If the properties in which we are interested do not depend on $\pi$, we will simply denote by $\mathcal{I} \ \Box \ \mathcal{C}$ the concatenated code. 

\subsection{Strong blocking sets}\label{sec:strongblockingsets}


For $k>1$, the \emph{finite projective geometry} of dimension $k-1$ and order $q$, denoted by $\PG(k-1,q)$ is
$$ \PG(k-1,q):= \left(\F_q^{k}\setminus \{0\}\right)/_\sim, $$
where $u\sim v$ if and only if $u=\lambda v$
for some nonzero $\lambda \in \F_q$.
A \emph{hyperplane} is a subspace $\HHH\subseteq \PG(k-1,q)$ of codimension $1$.  A \emph{projective $[n,k,d]_q$ system} $\mathcal{P}$ is a finite  set of $n$ points (counted with multiplicity) of $\PG(k-1,q)$ not all lying on a hyperplane and such that \[d = n- \max_{\HHH \text{ hyperplane }}\{|\HHH\cap \mathcal{P}|\}.\]

For an $[n,k,d]_q$  non-degenerate code $\mathcal{C}$ with generator matrix  $G$, let $\mathcal{P}$ be set of columns of $G$, considered as points in $\PG(k-1,q)$. For $u=(u_1,\ldots,u_k)\in \F_q^k$, $\wH(uG)=n-|\HHH\cap \mathcal{P}|$, where $\HHH$ is the hyperplane defined by the equation $u_1x_1+\ldots+u_kx_k=0$. So $\mathcal{P}$ is a projective $[n,k,d]_q$ system. For the same reason, from a projective $[n,k,d]_q$ system $\mathcal{P}$ one can obtain an $[n,k,d]_q$ code by choosing a representative for any point of $\mathcal{P}$ and considering the code generated by the matrix having these representatives as columns. This gives the well-known correspondence (see for example \cite[Theorem~1.1.6]{tsfasman1991algebraic}) between equivalence classes of non-degenerate $[n,k,d]_q$ linear codes and equivalence classes of projective $[n,k,d]_q$ systems.

Minimal codes correspond to  $[n,k,d]_q$ systems with additional properties.

\begin{definition}[\!\!\cite{1930-5346_2011_1_119}]
A subset $\mathcal{M} \subseteq \PG(k-1,q)$ is called a \emph{strong blocking set} if for every hyperplane $\HHH$ of $\PG(k-1,q)$ we have $\langle \mathcal{M} \cap \HHH \rangle =\HHH$.
 \end{definition} 

\begin{example}\label{example:wholespace}
The whole $\PG(k-1,q)$ is, clearly, a strong blocking set.
\end{example}

Strong blocking sets were introduced in \cite{1930-5346_2011_1_119}. In \cite{zbMATH06340223}, strong blocking sets are referred to as generator sets and they are constructed as union of disjoint lines. In \cite{bonini2020minimal} they were reintroduced, with the name of \emph{cutting blocking sets}, in order to construct a particular family of minimal codes. 
In \cite{alfarano2019geometric} and \cite{Full_Characterization} it was independently shown that strong blocking sets are the geometrical counterparts of minimal codes: the above correspondence between non-degenerate codes and projective systems restricts to a correspondence between equivalence classes of projective $[n,k,d]_q$ minimal codes and equivalence classes of subsets of $\PG(k-1,q)$ that are strong blocking sets. Since in this paper we are interested in short minimal codes or, equivalently, in  small strong blocking sets in projective spaces, it is not restrictive to narrow down to projective codes, which correspond to projective systems in which all the points have multiplicity one. Clearly, the simplex codes defined in Example \ref{example:simplexcode} correspond to the strong blocking sets defined in Example \ref{example:wholespace}. 

\subsection{Bounds on the size of strong blocking sets}

Given the above equivalence between minimal codes and strong blocking sets, for the sake of brevity, in this section we will only state the bounds for the size of the latter, but clearly these bounds also apply to the length of minimal codes. 

For a given dimension and a base field, simplex codes are the longest projective minimal codes, as the whole projective space is the largest strong blocking set. Since from a given strong blocking set we can always obtain a larger one by adding any set of points, it is interesting to know how small  a strong blocking set can be. 

A $(k-1)$-\emph{fold blocking set} in $\PG(k-1,q)$ is a subset $\mathcal{B}$ of points such that every hyperplane meets $\mathcal{B}$ in at least $k-1$ points. Clearly, a strong blocking set in  $\PG(k-1,q)$ is a $(k-1)$-fold blocking set. By a well-known result of Beutelspacher on $(k-1)$-fold blocking sets (see \cite[Theorem~2]{beutelspacher1983}), the size of a strong blocking set in $\PG(k-1,q)$ is at least $(q+1)(k-1)$, if $k\leq q+1$. The same result holds without restrictions on $k$:

\begin{theorem}[\!\!\cite{alfarano2020three}]\label{thm:lowerbound} The size of a strong blocking set in $\mathrm{PG}(k-1,q)$ is at least $(q+1)(k-1)$.
\end{theorem}

However, this lower bound is not sharp in general, as shown in \cite{alfarano2020three}.\\
There are many constructions of strong blocking sets of small size. One usual way to obtain them is to consider sets of lines: a set of lines of $\mathrm{PG}(k-1,q)$ is in higgledy-piggledy arrangement if the union of their point sets is a strong blocking set of $\mathrm{PG}(k-1,q)$. 
Using higgledy-piggledy lines one can get strong blocking sets of size $(q+1)(2k-3)$ in $\mathrm{PG}(k-1,q)$, whenever $ 2k\leq q+3$ (see \cite[Theorem 20]{zbMATH06340223}). Note that the above result cannot be used to obtain strong blocking sets for large dimensions. There are stronger results for small dimensions: see  \cite{zbMATH06340223,1930-5346_2011_1_119} for $k=3,4$, \cite{Bartoli_2020,Denaux} for $k = 5$,
and \cite{Bartoli:2020tk} for $k = 6$.

The best bounds concerning the size of smallest strong blocking sets are summarized in the following theorem, which is a refinement of a result in \cite{Chabanne_2014}.

\begin{theorem}[\!\!\cite{Heger:2021vr}]\label{thm:upperbound}
The size of the smallest strong blocking sets in $\mathrm{PG}(k-1,q)$ is at most
$$
\begin{cases}
\frac{2k-1}{\log_2(4/3)},& \textrm{ if } q=2;\\
(q+1)\left\lceil \frac{2}{1+\frac{1}{(q+1)^2\ln q}}(k-1)\right\rceil ,& \textrm{otherwise.}
\end{cases}$$
\end{theorem}

The proof of the above theorem is probabilistic. Explicit small constructions can be found in \cite{alfarano2020three,alfarano2019geometric}: they are of size $ck^2q$,
for some $c \geq 2/9$. However, from Theorem~\ref{thm:lowerbound} and Theorem~\ref{thm:upperbound} we know that a construction where the size is linear in $k$ (and $q$) does exist. Note that this would yield explicit construction of asymptotically good families of minimal codes, by the lower bounds on the minimum weight (see \cite[Theorem 2.8]{alfarano2020three}). 

In \cite[\S 2.A]{cohen-zemor} (for $q=2$) and \cite[\S 2.4]{MR3644879} (for all others cases), such a construction is sketched, and it employs the concatenation. However, these results seem to have gone unnoticed to date, in the context of both minimal codes and strong blocking sets. The main aim of this paper is to highlight the power of this approach by delving into it.

\section{Strong blocking sets by concatenation}\label{sec:strong}

\subsection{Geometric interpretation of the Ashikhmin-Barg condition}

First, we provide an easy  geometrical interpretation of part 3 of \cite[Lemma 2.1]{ashikhmin1998minimal}, known as the  Ashikhmin-Barg condition, which provides a sufficient condition for a subset of points to be a strong blocking set. Even if it is a direct consequence of the geometrical interpretation of minimal codes described above, we give a direct proof using only geometrical arguments.

\begin{lemma}[Ashikhmin-Barg] Let $\mathcal{B}$ be a subset of $\mathrm{PG}(k-1,q)$ of cardinality $n$ and denote by $m$ (resp. $M$) the minimum (resp. the maximum) number of points of $\mathcal{B}$ contained in a hyperplane of $\mathrm{PG}(k-1,q)$. If 
\begin{equation}\label{eq:ABgeom}
\frac{n-M}{n-m}>\frac{q-1}{q}
\end{equation}
then $\mathcal{B}$ is a strong blocking set.
\end{lemma}
\begin{proof}
Suppose that $\mathcal{B}\subseteq \mathrm{PG}(k-1,q)$ is not a strong blocking set. This means that not all the hyperplanes $\mathcal{H}$ are generated by their intersection $\mathcal{H}\cap \mathcal{B}$ with $\mathcal{B}$. Equivalently, there exists a hyperplane $\mathcal{H}$ such that $\mathcal{H}\cap \mathcal{B}$ is contained in a subspace $\mathcal{S}$ of codimension $2$. Recall that there are exactly $q+1$ distinct hyperplanes through a given subspace of codimension $2$. Consider the $q$ hyperplanes $\mathcal{H}_1,\ldots,\mathcal{H}_q$ through $\mathcal{S}$ distinct from $\mathcal{H}$. Let $x=|\mathcal{H}\cap \mathcal{B}|\geq m$. Thus $(\mathcal{H}_i\setminus \mathcal{S})\cap \mathcal{B}$ has size at most $M-x$ for each $i=1,\ldots,q$. Therefore 
$$n=|\mathcal{B}|\leq x+q(M-x)=qM-(q-1)x\leq qM-(q-1)m,$$
 a contradiction to \eqref{eq:ABgeom}.
\end{proof}

In coding theoretical language, \eqref{eq:ABgeom} reads as follows: let $\mathcal{C}$ be the (projective) linear code of length $n$ associated with $\mathcal{B}$. If 
\begin{equation}\label{eq:ABclass} \frac{d(\C)}{w(\C)}> \frac{q-1}{q}
\end{equation}
then $\mathcal{C}$ is minimal. In the sequel, we will refer to \eqref{eq:ABgeom} and \eqref{eq:ABclass} as AB condition. 

\subsection{Outer AB condition}

The concatenation process allows to get strong blocking sets by imposing a similar but weaker condition on the outer codes. 

\begin{theorem}[Outer AB condition]\label{th:meta}
Let $\C$ be an $[N,K,D]_{q^k}$  code such that $$\frac{D}{W}>\frac{q-1}{q},$$ where $W:=w(\C)$, and $\mathcal{I}$ be an $[n,k,d]_q$ minimal code. Then the concatenated code $\mathcal{I} \ \Box_\pi \ \mathcal{C}$ is a minimal code of parameters $[Nn,Kk,\geq Dd]_q$. Equivalently, the corresponding set in $\mathrm{PG}(Kk-1,q)$ is a strong blocking set of size $Nn$.
\end{theorem}
\begin{proof}
Let $x$ and $x'$ be two nonzero  codewords in $\mathcal{I} \ \Box_\pi \ \mathcal{C}$ such that $\sigma(x')\subseteq \sigma(x)$. Let $c$ and $c^{\prime}$ be the codewords in $\C$ such that  $x=(\pi(c_1),\ldots,\pi(c_N))$ and $x'=(\pi(c^{\prime}_1),\ldots,\pi(c^{\prime}_N))$. Clearly $\sigma(c^{\prime})\subseteq \sigma(c)$. Moreover, since $\mathcal{I}$ is minimal, there exists $\lambda_i\in \F_q^*$ such that $c_i=\lambda_i c^{\prime}_i$ whenever $i\in \sigma(c^{\prime})$. Since $\wt(c^{\prime})\geq D$, the scalars are equal to a certain $\lambda\in \F_q^*$ in at least $\left\lceil \frac{D}{q-1}\right\rceil $ coordinates. So, up to a multiplication by $\lambda^{-1}$, $c$ and $c^{\prime}$ coincide in at least $\left\lceil \frac{D}{q-1}\right\rceil $ coordinates.

Now 
$$D\leq \wt(c^{\prime})\leq \wt(c)\leq W$$
and, since $D/W>(q-1)/q$,
$$W<\frac{q}{q-1}\cdot D=D+\frac{D}{q-1}\leq  D+\left\lceil \frac{D}{q-1}\right\rceil.  $$
Thus 
$$\wt(c-c^{\prime})\leq W-\left\lceil \frac{D}{q-1}\right\rceil<D,$$
which yields $\wt(c-c^{\prime})=0$ and $c=c^{\prime}$. Hence $\mathcal{I} \ \Box_\pi \ \mathcal{C}$ is minimal.
\end{proof}

\begin{remark}
As one can easily see from the proof of Theorem \ref{th:meta}, we may get the same result by concatenating in each coordinate with a different minimal code.
\end{remark}

It is possible to provide a more geometrical interpretation of the concatenation process. In this geometrical description of concatenated codes, companion matrices play a crucial role. Let 
 $p(x)=x^k+\sum_{i=0}^{k-1}p_ix^i\in \mathbb{F}_q[x]$ be an irreducible monic polynomial of positive degree $k$ and define the companion matrix of $p(x)$ as 
$$A:=
\begin{pmatrix}
0&1&0&\cdots&0\\
0&0&1&\cdots&0\\
\vdots&\vdots&\vdots&&\vdots\\
0&0&0&\cdots&1\\
-p_0&-p_1&-p_2&\cdots&-p_{k-1}\\
\end{pmatrix}.$$
The $\mathbb{F}_q$-algebra $\mathbb{F}_q[A]$ is a finite field with $q^k$ elements (see \cite[Chapter 2.5]{lidl_niederreiter_1994}). In particular, if $\alpha$ is a root of $p(x)$ and $v=v_0+v_1\alpha+\ldots+v_{k-1}\alpha^{k-1}$ is a generic element of $\F_q[\alpha]\cong \F_{q^k}$ and $\phi:\F_q[\alpha]\to \F_q^k$ is the $\F_q$-linear isomorphism that sends $v$ to $[v_0,\ldots,v_{k-1}]$, then $\phi(v)A=\phi(v\alpha)$.

\begin{remark}[Geometric insight of Theorem \ref{th:meta}]
Let $G_{\mathcal{I}}\in \mathbb{F}_{q}^{k\times n}$ be a generator matrix of the inner code $\mathcal{I}$ and let $\mathbb{F}_{q^k}^*=\langle \omega\rangle$. Denote by $A\in \mathbb{F}_q^{k\times k}$ the companion matrix of the minimal polynomial of $\omega$. For an element $\alpha \in \mathbb{F}_{q^k}$ let 
$$A(\alpha):= 
\begin{cases}
A^{r} \in \mathbb{F}_q^{k\times k},& \textrm{if } 0\neq \alpha=\omega^r,\\
{\bf 0} \in \mathbb{F}_q^{k\times k},& \textrm{if } \alpha= 0.
\end{cases}
$$
Consider now a generator matrix $G_{\mathcal{C}}\in\mathbb{F}_{q^k}^{K\times N}$ of the outer code $\mathcal{C}$
$$
G_{\mathcal{C}}=
\begin{pmatrix}
\alpha_{1,1}&\alpha_{1,2}&\cdots &\alpha_{1,N}\\
\alpha_{2,1}&\alpha_{2,2}&\cdots &\alpha_{2,N}\\
\vdots& \vdots& &\vdots\\
\alpha_{K,1}&\alpha_{K,2}&\cdots &\alpha_{K,N}\\
\end{pmatrix}.
$$
A generator matrix of a concatenated code $\mathcal{I} \ \Box_\phi \ \mathcal{C}$ can be chosen as follows 
\begin{equation}\label{eq:G}
\begin{pmatrix}
A(\alpha_{1,1})G_{\mathcal{I}}&A(\alpha_{1,2})G_{\mathcal{I}}&\cdots &A(\alpha_{1,N})G_{\mathcal{I}}\\
A(\alpha_{2,1})G_{\mathcal{I}}&A(\alpha_{2,2})G_{\mathcal{I}}&\cdots &A(\alpha_{2,N})G_{\mathcal{I}}\\
\vdots& \vdots& &\vdots\\
A(\alpha_{K,1})G_{\mathcal{I}}&A(\alpha_{K,2})G_{\mathcal{I}}&\cdots &A(\alpha_{K,N})G_{\mathcal{I}}\\
\end{pmatrix}\in \mathbb{F}_q^{Kk\times Nn}.
\end{equation}

Let 
\begin{itemize}
    \item $\mathcal{P}:=\{P_{x+ny} \mid x\in\{1,\ldots,n\}, y\in \{0,\ldots,N-1\}\}$ be the set of points in $\mathrm{PG}(Kk-1,q)$ corresponding to the columns of \eqref{eq:G};
    \item $\mathcal{Q}:=\{Q_1,\ldots,Q_n\}$ be the set of points in $\mathrm{PG}(k-1,q)$  corresponding to the columns of $G_\mathcal{I}$;
    \item $\mathcal{R}:=\{R_1,\ldots,R_N\}$ be the set of points in $\mathrm{PG}(K-1,q^k)$ corresponding to the columns of $G_\C$.
\end{itemize}

Let $\mathcal{H}$ and $\mathcal{H}^{\prime}$ be two hyperplanes of $\mathrm{PG}(Kk-1,q)$ such that $\mathcal{P}\cap \mathcal{H} \subseteq \mathcal{P}\cap \mathcal{H}^{\prime}$. We aim to prove that  $\mathcal{H}=\mathcal{H}'$, which is equivalent to $\mathcal{P}$ being a strong blocking set (see for example \cite[Proposition 3.3.]{alfarano2019geometric}) 

Let $v$ and $v'$ be two vectors in $\F_q^{Kk}$ such that $\mathcal{H}=\langle v\rangle^\perp$ and $\mathcal{H}'=\langle v'\rangle^\perp$ and let
$$v=(v_1|\ldots |v_K), \ \ \ v'=(v'_1|\ldots |v'_K),$$
with $v_i,v'_i\in \F_q^k$. In the following, if $\mathcal{H}=\langle v\rangle^\perp$,  we denote by $\mathcal{H}(P)$  the value $\sum_{i=1}^k v_i P_i $, where $P=(P_1,\ldots,P_K)$.

Consider $y\in \{0,\ldots,N-1\}$.

First we want to prove that there exists $\lambda_y \in \mathbb F_q^*$ such that
\[ \sum_{i=1}^K v_i A(\alpha_{i,y+1}) = \lambda_y \sum_{i=1}^K v'_i A(\alpha_{i,y+1}).\]

To this aim, consider for any $x\in \{1,\ldots,n\}$
$$\mathcal{H}(P_{x+ny})=\sum _{i=1}^{K}\underbrace{v_i}_{\in \mathbb{F}_q^{1\times k}}\underbrace{A(\alpha_{i,y+1})}_{\in \mathbb{F}_q^{k\times k}}\underbrace{Q_x}_{\in \mathbb{F}_q^{k\times 1}}=\left(\sum _{i=1}^{K}v_i A(\alpha_{i,y+1}) \right) Q_x$$
and similarly 
$$\mathcal{H}'(P_{x+ny})=\sum _{i=1}^{K}\underbrace{v'_i}_{\in \mathbb{F}_q^{1\times k}}\underbrace{A(\alpha_{i,y+1})}_{\in \mathbb{F}_q^{k\times k}}\underbrace{Q_x}_{\in \mathbb{F}_q^{k\times 1}}=\left(\sum _{i=1}^{K}v'_i A(\alpha_{i,y+1}) \right) Q_x.$$
 Now,
$$\left(\sum _{i=1}^{K}v_i A(\alpha_{i,y+1}) \right) \textrm{ and } \left(\sum _{i=1}^{K}v'_i A(\alpha_{i,y+1}) \right) $$
both correspond (if non-vanishing) to  hyperplanes $\pi$ and $\pi^{\prime}$ in $\mathrm{PG}(k-1,q)$.

Since $\mathcal{P}\cap \mathcal{H} \subseteq \mathcal{P}\cap \mathcal{H}^{\prime}$, if  $\mathcal{H}(P_{x+ny})=0$ then $\mathcal{H}'(P_{x+ny})=0$ too. Thus, $(\pi\cap \mathcal{Q})\subset (\pi^{\prime}\cap \mathcal{Q})$. 
Recall that $\mathcal{Q}$ is a strong blocking set in $\mathrm{PG}(k-1,q)$ and therefore there exists $\lambda_y\in \mathbb{F}_q^*$ such that 
\[ \sum_{i=1}^K v_i A(\alpha_{i,y+1}) = \lambda_y \sum_{i=1}^K v'_i A(\alpha_{i,y+1}).\]

Writing $v_i = \phi(\eta_i)$ and $v'_i = \phi(\eta'_i)$, this implies that
\begin{align*}
\sum_{i=1}^K \phi(\eta_i \alpha_{i,y+1}) = \lambda_y \sum_{i=1}^K \phi( \eta'_i \alpha_{i,y+1})
&&\Longrightarrow&&
\sum_{i=1}^K \eta_i \alpha_{i,y+1} = \lambda_y \sum_{i=1}^K \eta'_i \alpha_{i,y+1},
\end{align*}
 by the $\mathbb F_q$-linearity and injectivity of $\phi$.

Now we consider $\eta=(\eta_1,\ldots,\eta_K)$, $\eta^{\prime}=(\eta_1^{\prime},\ldots,\eta_K^{\prime})$ and   the hyperplanes $\overline{\mathcal{H}}=\langle \eta\rangle^\perp$, $\overline{\mathcal{H}}'=\langle \eta'\rangle^\perp$ in $\mathrm{PG}(K-1,q^k)$ corresponding to $\eta$ and $\eta'$ respectively.

Suppose that there is no $\lambda_y\in \mathbb{F}_q^*$ such that $\eta=\lambda_y \eta^{\prime}$. Then there is $\lambda_y$ such that the hyperplane  $(\eta-\lambda_y \eta^{\prime})^{\perp}$ has at least $\lceil N/(q-1)\rceil$ points in common with $\mathcal{R}$. By hypothesis, the maximum number of points of $\mathcal{R}$ on a hyperplane is $N-D$ and thus $N/(q-1)\leq N-D$, which gives $D/N\leq (q-2)/(q-1)$, a contradiction to $D/N> (q-1)/q$. 

This means that $\eta=\lambda_y \eta^{\prime}$ for some $\lambda_y\in \mathbb{F}_q^*$ and  so  $\mathcal{H}=\mathcal{H}'$ and hence the set $\mathcal{P}$ is a strong blocking set.
\end{remark}

\begin{remark}
We can reinterpret the construction in \cite[Theorem 4.1.]{alfarano2020three} in terms of concatenation. Actually, if $\omega$ is a primitive element of $\F_{q^k}$, that construction is the concatenation of a $[4,2,3]_{q^k}$ MDS code $\C$ with generator matrix
$$G_\C=\begin{bmatrix}1&1&1&0\\0&\omega^i&\omega^j&1\end{bmatrix},$$
where $0<i<j<q^k$, with a simplex code $\mathcal{S}_q(k)$. Theorem \ref{th:meta} affirms that $\mathcal{S}_q(k) \ \Box \ \C$ is minimal if $q<4$. However, by \cite[Theorem  4.1.]{alfarano2020three} we have that $\mathcal{S}_q(k) \ \Box \ \C$ is minimal whenever $i-j\not\equiv 0 \bmod \frac{q^s-1}{q-1}$ for every $s>1$ dividing $k$. This provides an infinite family of minimal codes obtained with concatenation, where the outer codes do not satisfy the Outer AB condition.
\end{remark}

In view of Theorem \ref{th:meta}, it is important to have explicit constructions of codes satisfying the Outer AB condition. This is the case, for example, of some MDS codes, as the following result shows.

\begin{corollary}\label{thm:MDSconcat}
Let $k$ be an integer greater than $1$. For $K\leq q^{k-1}+1$ and $N:=qK-(q-1)$, let $\C$ be an $[N,K,(q-1)(K-1)+1]_{q^k}$ MDS code (which exists under this hypothesis on $K$) and $\mathcal{I}$ be an $[n,k,d]_q$ minimal code. Then the concatenation $\mathcal{I} \ \Box \ \mathcal{C}$ of $\C$ with $\mathcal{I}$ is a  $[(qK-q+1)n,Kk,\geq ((q-1)(K-1)+1)d]_q$ minimal code.

Thus, there exists a strong blocking set in $\mathrm{PG}(Kk-1,q)$ of size $\displaystyle nqK-qn+n$, for any $K\leq q^{k-1}+1$.
\end{corollary}
\begin{proof}
It follows directly from Theorem \ref{th:meta}, since, with the same notations as in Theorem \ref{th:meta},
$$\frac{D}{W}\geq \frac{D}{N}=\frac{(q-1)(K-1)+1}{N}>\frac{q-1}{q}.$$
\end{proof}

If one is interested in getting small strong blocking sets, the shorter the inner codes are, the better it is. On the other hand, concatenating with simplex codes has the great advantage that it is easy to get the weight distribution of the concatenated code. Actually, if $A_i(\C)=|\{c\in \C\mid \wt(c)=i\}|$, then
$$A_{q^{k-1}i}(\mathcal{S}_q(k) \ \Box \ \mathcal{C})=|\{c\in \mathcal{S}_q(k)  \ \Box \ \mathcal{C}\mid \wt(c)=q^{k-1}i\}|=A_i(\C),$$
and $A_{j}(\mathcal{S}_q(k) \ \Box \ \mathcal{C})=0$ for all $j$ which are not multiples of $q^{k-1}$.

\begin{corollary}\label{Th:metaconstruction}
Let $\mathcal{C}$ be an $[N,K,D]_{q^k}$ linear code with maximum weight $W$. Suppose that $D/W>(q-1)/q$. Then there exists a minimal $[N(q^k-1)/(q-1),Kk,Dq^{k-1}]_q$ code $\mathcal{D}$. Equivalently, in $\mathrm{PG}(Kk-1,q)$ there exists a strong blocking set of size  $N(q^k-1)/(q-1).$
\end{corollary}
\begin{proof}
This is a straightforward consequence of Theorem \ref{th:meta}, recalling that the simplex code of parameters $[(q^k-1)/(q-1),k,q^{k-1}]_{q}$ is in particular a minimal code.
\end{proof}

\begin{remark}
When concatenating with a simplex code of dimension $k$, it is worth noting that the outer code in our construction is required have relative distance larger than $1-1/q$. Therefore  its rate is smaller than $1/q+1/n$ and thus the rate of the concatenated code is smaller than  
$$\frac{k(q-1)}{q^k-1}\left(\frac{1}{q}+\frac{1}{n}\right)\leq \frac{2}{q+1}\left(\frac{1}{q}+\frac{1}{n}\right).$$
The size of the corresponding strong blocking set is lower bounded by 
$$\frac{q(q+1)}{2}\cdot k-q.$$
One can, in principle, search for outer codes having a slightly larger relative minimum weight, say $(1-1/q)^{1+\epsilon}$, and concatenate them with an inner code with relative distance (or ratio between minimum and maximum weight) at least $(1-1/q)^{-\epsilon}$ and rate  larger than $\frac{k(q-1)}{q^k-1}$. In this way one obtains constructions of  shorter minimal codes and so of smaller strong blocking sets. 
\end{remark}

\section{Asymptotic results}\label{sec:asymptotic}

\subsection{Non-constructive}

By the Gilbert-Varshamov Bound (see \cite[\S 1.3.2]{tsfasman1991algebraic}), for any $0 \leq \delta \leq 1 -1/q$ and $0<\gamma \leq 1 - H_q(\delta)$ there exists a $q$-ary linear code of relative minimum weight $\delta$ and rate at least $1 - H_q(\delta)-\gamma$, where $$H_q(x) := x \log_q(q - 1) - x \log_q(x) - (1 - x) \log_q(1 -x).$$ 

The following theorem shows that it is possible to construct for any $q$ an infinite family of strong blocking sets of size linear in the dimension $k$.

\begin{theorem}
For any prime power $q>2$ and any $\epsilon \in (0,1-1/\ln q)$, there exists an integer $k_{q,\epsilon}$ such that  $\mathrm{PG}(k_{q,\epsilon}-1,q)$ possesses a strong blocking set of size smaller than
$$k_{q,\epsilon}\cdot  \frac{q^{1+\epsilon}(q+1)}{\left(1-\epsilon -\frac{1}{\ln q}\right)}.$$
\end{theorem}
\begin{proof}
Fix $\delta = 1 -1/q^{1+\epsilon}\leq 1-1/q^2$, for $\epsilon\in (0,1-1/\ln q)$. By the Gilbert-Varshamov Bound, with respect to 
$$\gamma<\min \left\{1-H_{q^2}(\delta),\left(1-\frac{1}{q^{1+\epsilon}}\right)\left(1-\frac{\log_q(q^2-1)}{2}\right)\right \}, $$ there exists  a $q^2$-ary linear code $\mathcal{C}$ of dimension $k_{q,\epsilon}$ with relative minimum weight $\delta$ and rate at least

\begin{eqnarray*}
1 - H_{q^2}(\delta)-\gamma&=&1-\left(1 -\frac{1}{q^{1+\epsilon}}\right) \log_{q^2}\frac{q^2 - 1}{\left(1 -\frac{1}{q^{1+\epsilon}}\right)} +\frac{1}{q^{1+\epsilon}}\log_{q^2}\left(\frac{1}{q^{1+\epsilon}}\right)-\gamma\\
&=&1-\left(1 -\frac{1}{q^{1+\epsilon}}\right) \left( \frac{\log_q (q^2-1)}{\log_q q^2}-\frac{\log_q \left(1 -\frac{1}{q^{1+\epsilon}}\right)}{\log_q q^2}\right) -\frac{1+\epsilon}{2q^{1+\epsilon}}-\gamma\\
&\geq &1-\left(1 -\frac{1}{q^{1+\epsilon}}\right) \left( 1-\frac{\log_q \left(1 -\frac{1}{q^{1+\epsilon}}\right)}{2}\right) -\frac{1+\epsilon}{2q^{1+\epsilon}}\\
&= &\frac{1-\epsilon}{2q^{1+\epsilon}}+\frac{1}{2}\left(1 -\frac{1}{q^{1+\epsilon}}\right) \log_q \left(1 -\frac{1}{q^{1+\epsilon}}\right)\\
&\geq  &\frac{1}{2q^{1+\epsilon}}\left(1-\epsilon -\frac{1}{\ln q}\right)>0,\\
\end{eqnarray*}
since 
$$\left(1 -\frac{1}{q^{1+\epsilon}}\right) \log_q \left(1 -\frac{1}{q^{1+\epsilon}}\right)=\frac{1}{\ln q}\left(1 -\frac{1}{q^{1+\epsilon}}\right) \ln \left(1 -\frac{1}{q^{1+\epsilon}}\right)>-\frac{1}{q^{1+\epsilon}\ln q}.$$
By Corollary~\ref{Th:metaconstruction}, there exists a minimal $q$-ary code of rate larger than
$$\frac{2}{q+1}\cdot \frac{1}{2q^{1+\epsilon}}\left(1-\epsilon -\frac{1}{\ln q}\right)= \frac{1}{q^{1+\epsilon}(q+1)}\left(1-\epsilon -\frac{1}{\ln q}\right)$$
and therefore a strong blocking set in $\mathrm{PG}(k_{q,\epsilon}-1,q)$ of size smaller than
$$k_{q,\epsilon}\cdot  \frac{q^{1+\epsilon}(q+1)}{\left(1-\epsilon -\frac{1}{\ln q}\right)} .$$
\end{proof}

\subsection{Explicit constructions}\label{subsection:explicit constructions}

As we have just seen, Corollary~\ref{Th:metaconstruction} and the Gilbert-Varshamov bound imply the existence of strong blocking set of size linear in the dimension and quadratic in size of the field. In this subsection, exploiting explicit constructions of asymptotically good AG codes, we provide constructions of families of asymptotically good minimal codes and therefore infinite families of small strong blocking sets. 

Let $N_q(\mathcal{X})$ denote the number of degree 1 places of an $\mathbb{F}_q$-rational curve $\mathcal{X}$ and consider the Ihara's constant
$$A(q)=\limsup_{g\to \infty} \frac{\max \{\#N_q(\mathcal{X}) \ | \ \mathcal{X} \textrm{ has genus }  g\}}{g}. $$
By the Drinfeld-Vladut Bound (see \cite[Theorem 2.3.22]{tsfasman1991algebraic}) we know that $$A(q) \leq \sqrt{q} - 1.$$

If $q$ is a square then the Drinfeld-Vladut Bound is actually attained, as shown by  Ihara \cite{zbMATH03803619} and
Tsfasman, Vladut and Zink \cite{zbMATH03918262}, using the theory of modular curves. 

We apply now Corollary~\ref{Th:metaconstruction} to the family of AG codes obtained by the curves considered in \cite{MR1423052}.

\begin{theorem}\label{Th:main}
Let $q_0$ be a prime power. Consider  $h\geq 2$. 
Denote by 
\begin{eqnarray*}
N_{h,n}&:=& (q_0^{h(n+1)}-q_0^{hn}-1);\\
\lambda_{h,n}&:=&\begin{cases}
(q_0^{nh/2}-1)^2,& \textrm{if } 2\mid n;\\
(q_0^{(n+1)h/2}-1)(q_0^{(n-1)h/2}-1),& \textrm{if } 2\nmid n.\\
\end{cases}\\ 
\end{eqnarray*}
There exists a family  of good minimal codes $\mathcal{C}_{h,n}$ with parameters

$$\left[N_{h,n}\frac{q_0^{2h}-1}{q_0-1},2h\left(\left\lfloor\frac{N_{h,n}-1}{q_0} \right\rfloor-\lambda_{h,n}+1\right),\geq q_0^{{2h}-1}N_{h,n}\left(1-\frac{1}{q_0}\right)\right]_{q_0}.$$
\end{theorem}
\begin{proof}
Denote $q=q_0^h>2$.  Consider the tower $\mathcal{T}=(T_1,T_2,T_3,\ldots)$ of function fields over $\F_{q^2}$ given by $T_n:=\F_{q^2}(x_1,\ldots,x_n)$, with 
\[x_{i+1}^q+x_{i+1}=\frac{x_i^q}{x_i^{q-1}+1}, \ \text{for} \ i\in\{1,\ldots,n-1\}. \]
By \cite[Remark 3.8.]{MR1423052}, the genus of $T_n$ is
\[g_n:=g(T_n)=\left\{\begin{array}{ll}
(q^{n/2}-1)^2, & \text{if } n\equiv 0 \bmod 2;\\
(q^{(n+1)/2}-1)(q^{(n-1)/2}-1), & \text{if } n\equiv 1 \bmod 2.\\
\end{array}\right.\]
Consider
\begin{eqnarray*}
\Omega&:=&\{\alpha\in \F_{q^2}\mid \alpha^q+\alpha=0\};\\
S&:=&\{R_\alpha\in \mathbb{P}(T_1) \mid \alpha\notin \Omega\}.
\end{eqnarray*}
By \cite[Lemma 3.9]{MR1423052}, any $R\in S$ splits completely in all the extensions $T_n/T_1$. Let $S_n$ be the places in $T_n$ lying over the places in $S$. Since $[T_n:T_1]=q^{n-1}$, 
$\#S_n=\#S\cdot [T_n:T_1]=q^{n-1}(q^2-q)$.
Let $n_n:=\#S_n-1=q^{n-1}(q^2-q)-1$.


Take a place $P_n\in S_n$ and the divisor $G_n:=m_nP_n$ with $2g_n-2<m_n<n_n$. Since $q>2$, such a value $m_n$ exists. Then 
\[k_n:=\dim\C_n=\ell(G_n)=m_n-g_n+1,\]
where  the last equality is the celebrated Riemann-Roch Theorem (see \cite[Chapter 2]{tsfasman1991algebraic}).

In what follows we consider the AG code $\mathcal{C}_n:= C\left(G_n,\sum_{P\in S_n\setminus \{P_n\}}P\right)$, which is an $[n_n,m_n-g_n+1,d_n\geq n_n-m_n]_{q^2}$-code.

Note that
$$\frac{d_n}{w_n}\geq \frac{d_n}{n_n}\geq \frac{n_n-m_n}{n_m}=1-\frac{m_n}{n_n}.$$
Consider $m_n=\left\lfloor\frac{n_n-1}{q_0} \right\rfloor $, so that
$\frac{d_n}{w_n}\geq 1-\frac{m_n}{n_n}>1-\frac{1}{q_0}$.

By Corollary~\ref{Th:metaconstruction}  there exists a minimal $[n_n(q_0^{2h}-1)/(q_0-1),2k_nh,\geq q_0^{{2h}-1}n_n(1-1/q_0)]_{q_0}$-code $\mathcal{D}_n$.

The family $\mathcal{D}_n$ is asymptotically good since the relative minimum weight and the rate are larger than  
\begin{equation*}
\frac{q_0^{2h-2}(q_0-1)^2 }{q_0^{2h}-1} 
\end{equation*}
and 
\begin{eqnarray*}
\frac{2h(q_0-1)\cdot (m_n-g_n+1)}{(q_0^{2h}-1)\cdot n_n}  &>&\frac{2h(q_0-1)}{q_0^{2h}-1}\cdot \left(\frac{1}{q_0}-\frac{g_n}{n_n}\right)\\
&\geq& \frac{2h(q_0-1)}{q_0^{2h}-1}\cdot \left(\frac{1}{q_0}-\frac{1}{q_0^h-1}\right),\\
\end{eqnarray*}
whenever $n\geq 3$. Note that, the assumption $h\geq 2$ is necessary to have $\frac{1}{q_0}-\frac{1}{q_0^h-1}>0.$
\end{proof}

\begin{corollary}\label{Cor:particular}
Let $K_n:=4\left\lfloor\frac{N_{2,n}-1}{q_0}\right\rfloor-4\lambda_{2,n}+4$. In $\mathrm{PG}(K_n-1 ,q_0)$ there exists a strong blocking sets of size at most  $$\left(\frac{q_0(q_0^4-1)(q_0+1)}{4(q_0^2-q_0-1)}\right)\cdot K_n \leq q_0\cdot \frac{q_0^3+2q_0^2+3q_0+5}{4}\cdot K_n.$$
\end{corollary}
\begin{proof}
The claim follows considering $h=2$ in Theorem \ref{Th:main}. In fact,
\begin{eqnarray*}
K_n&=&4\left\lfloor\frac{N_{2,n}-1}{q_0}\right\rfloor-4\lambda_{2,n}+4\\
&\geq& 4q_0^{2n+1}-4q_0^{2n-1}-4-4(q_0^n-1)^2+4\\
&=&4q_0^{2n+1}-4q_0^{2n}-4q_0^{2n-1}+8q_0^n -4\\
&\geq &4q_0^{2n+1}-4q_0^{2n}-4q_0^{2n-1},
\end{eqnarray*}
and so 
\begin{eqnarray*}\frac{N_{2,n}\frac{q_0^{4}-1}{q_0-1}}{K_n}&\leq& \frac{(q_0^{2n+2}-q_0^{2n}-1)(q_0^{4}-1)}{(q_0-1)(4q_0^{2n+1}-4q_0^{2n}-4q_0^{2n-1})}\\
&\leq& \frac{q_0(q_0^{2}-1)(q_0^{4}-1)}{4(q_0-1)(q_0^{2}-q_0-1)}\leq \frac{q_0(q_0^3+2q_0^2+3q_0+5)}{4}.
\end{eqnarray*}
\end{proof}

\begin{example}
Let $q_0=2$ and $h=3$, so that $q=8$. We consider the function field $T=\F_{64}(x,y)$ with $y^2+y=\frac{x^2}{x+1}$. The genus of $T$ is $49$. In this case the set $S$ of points in $T$ lying over the places that split completely in $T/\F_{64}(x)$ has cardinality $448$. Take any $P\in S$. 
Consider $m=223$, so that the divisor $G=223 P$ and $k=223-49+1=175$ (but any $k$ between $49$ and $175$ may also be chosen). The AG code $\mathcal{C}=C(G,\sum_{Q\in S\setminus\{P\}Q})$ is a $[447,175,\geq 224]_{64}$. Next, we concatenate with the $[7,3,4]_2$ simplex code, so that we obtain a minimal $[3129,525,\geq 896]_2$ code.
\end{example}

Note that, following the proof of Theorem \ref{Th:main}, it is readily seen that the sum of the rate and the relative distance of the concatenation of $\mathcal{C}_n$ with  the simplex code  is lower bounded by 
\begin{equation}\label{Eq:simplex}
1-\frac{2}{q_0}+O\left(\frac{1}{q_0^2}\right),
\end{equation}

In order to maximize it one can choose inner codes different from the simplex one,  as shown in the next example.

\begin{example}\label{Ex:RT4}
The RT4 in \cite[Figure 1b]{Calderbank_Kantor} is a two weight code over $\mathbb{F}_{q_0}$ with parameters $\widetilde{n}=(q_0^5-1)(q_0^2+1)/(q_0-1)$, $h=10$, $\widetilde{d}=q_0^6$, $\widetilde{w}=q_0^6+q_0^4$. Also, $\frac{\widetilde{d}}{\widetilde{w}}=\frac{q_0^2}{q_0^2+1}=1-\frac{1}{q_0^2+1}>1-\frac{1}{q_0}$. 

Now, choosing $m_n=\left\lfloor \frac{n_n(q_0^2-q_0+1)-1}{q_0^3}\right\rfloor$, one gets that the  concatenation of $\mathcal{C}_n$ (as in Theorem~\ref{Th:main}) with  the code RT4 is still minimal over $q_0$.  Also, the rate of $\mathcal{D}_n$ is lower bounded by 
\begin{equation*}
\frac{h}{\widetilde{n}q}\cdot \left(\left(1-\frac{\widetilde{w}}{\widetilde{d}}\frac{q_0-1}{q_0}\right)(q-1)-1\right)=\frac{10(q_0-1)}{q_0^{10}(q_0^5-1)(q_0^2+1)}(q_0^7(q_0^2-q_0+1)-2)\simeq \frac{10}{q_0^{7}}, 
\end{equation*}
which is less than the corresponding value for the simplex code in dimension $10$. 
Note that the relative minimum weight is larger than $\frac{q_0^3-q_0^2+q_0-1}{q_0^3},$
and 
$$\frac{q_0^3-q_0^2+q_0-1}{q_0^3}+\frac{10}{q_0^{7}}=1-\frac{1}{q_0}+\frac{1}{q_0^2}+o(q_0^{-2})$$
which is larger than \eqref{Eq:simplex}.
\end{example}


\begin{theorem}
Suppose that there exists an  $[\widetilde{n},2h,d]_{q_0}$ code $\mathcal{I}$, and denote by $\widetilde{D}:=\frac{d}{w}$. Let 
$$0<\epsilon< 1-\frac{1}{q_0^h-1}-\frac{1}{\widetilde{D}}\left(1-\frac{1}{q_0} \right).$$

Then there exists a family of minimal codes with rate at least $\frac{2h\epsilon}{\widetilde n }$. 

\end{theorem}
\begin{proof}
Consider the code $\mathcal{C}_n$ as in Theorem \ref{Th:main}. Choose $n$ and $m_n$ such that $$\epsilon+\frac{1}{q_0^h-1}-\frac{1}{n_n}\leq \frac{m_n}{n_n}<1-\frac{1}{\widetilde{D}}\left(1-\frac{1}{q_0}\right).$$
By our assumption on $\widetilde{D}$, such a value $m_n$ exists. Since
$$\left(1-\frac{m_n}{n_n}\right)\widetilde{D}>1-\frac{1}{q_0},$$
$\mathcal{I} \ \square \ \mathcal{C}_n$ is a minimal $q_0$-ary code by the AB condition.

Therefore,
$$\mathcal{R}(\mathcal{I} \ \square \ \mathcal{C}_n)\geq \frac{2h}{\widetilde{n}}\left(\frac{1}{n_n}+\frac{m_n}{n_n}-\frac{g_n}{n_n}\right)\geq \frac{2h\epsilon}{\widetilde{n}}.$$
\end{proof}

\begin{corollary}\label{Cor:General}
Suppose that there exists an  $[\widetilde{n},2h,d]_{q_0}$ code $\mathcal{I}$, and denote by $\widetilde{D}:=\frac{d}{w}$. Let 
$$0<\epsilon< 1-\frac{1}{q_0^h-1}-\frac{1}{\widetilde{D}}\left(1-\frac{1}{q_0} \right).$$

Let 
$$K_{h,n}:=\left\lfloor N_{h,n}\left(1-\frac{1}{\widetilde{D}}\left(1-\frac{1}{q_0}\right)\right)-1\right\rfloor.$$
In $\mathrm{PG}(K_{h,n}-1,q)$ there exists a strong blocking set of size at most 
$ \frac{\widetilde{n}}{2h\epsilon} K_{h,n}$.
\end{corollary}

\begin{remark}
Corollary \ref{Cor:General} provides an upper bound on the minimum size of strong blocking sets. Note that since $\widetilde{D}\geq 1$, $\epsilon <\frac{1}{q_0}-\frac{1}{q_0^h-1}$. 

On the other hand, by \cite[Corollary 3.7]{alfarano2020three}, since  $\mathcal{I}$ is not a constant weight code,
 $$\widetilde{n}\left(\frac{q_0^{2h-1}-1}{q_0^{2h}-1}\right)>2h-1\iff \frac{2h}{\widetilde{n}}<\frac{q_0^{2h-1}-1}{q_0^{2h}-1}+\frac{1}{\widetilde{n}}\iff \frac{\widetilde{n}}{2h}>\frac{\widetilde{n}(q_0^{2h}-1)}{\widetilde{n}(q_0^{2h-1}-1)+q_0^{2h}-1}.$$
The minimum  for the upper bounds on minimal strong  blocking set (considering also $\epsilon\simeq \frac{1}{2q_0}$) is roughly
 $$\frac{\widetilde{n}(q_0^{2h}-1)q_0}{\widetilde{n}(q_0^{2h-1}-1)+q_0^{2h}-1}K_{h,n}\simeq q_0^2K_{h,n},$$
 whenever $\widetilde{n}>>q_0$. This shows that in principle there is room for improvement in the bound provided by Corollary \ref{Cor:particular} if a suitable code $\mathcal{I}$ exists.
\end{remark}

\begin{remark}
It would be interesting to search for $[n,2h,d]_{q_0}$ codes $\mathcal{I}$ maximizing the quantity
$$\frac{2h}{n}\left(1-\frac{1}{q_0^h-1}-\frac{w}{d}\left( 1-\frac{1}{q_0}\right)\right) .$$
\end{remark}

\section{Numerical results and application to saturating sets}

\subsection{Some examples}

We conclude the paper with some numerical results and explicit constructions for small fields. This serves to document the power of the concatenation construction, which often provides the smallest strong blocking sets known, for a given dimension and field size.

We may concatenate MDS codes over $\F_{q^2}$ with  simplex codes $\mathcal{S}_q(2)$, of parameters $[q+1,2,q]_q$, which are the shortest of dimension $2$. We get $[(qK-q+1)(q+1),2K,\geq ((q-1)(K-1)+1)q]_q$ minimal codes. Those which are shorter than the upper bound of Theorem \ref{thm:upperbound} have the following parameters: $[9,4]_2$, $[15,6]_2$, $[16, 4]_3$, $[28, 6]_3$, $[40,8]_3$, $[25,4]_4$, $[45,6]_4$, $[65,8]_4$, $[85,10]_4$.

Another possibility is to obtain short minimal codes by concatenating MDS codes over $\F_{q^3}$ with short minimal codes of dimension $3$ over $\F_q$ (in ${\rm PG}(3,q)$ there are strong blocking sets of size $3q$,
namely the union of three non-concurrent lines). We get minimal codes shorter than the upper bound of Theorem \ref{thm:upperbound} for the following parameters: $[18, 6]_2$, $[30, 9]_2$, $[42, 12]_2$, $[54, 15]_2$, $[16, 6]_3$, $[28, 9]_3$.

Finally, let us remark that we use  inner codes corresponding to the tetrahedron (see \cite[Theorem 5.3.]{alfarano2019geometric}) or some shorter ones constructed in \cite{alfarano2020three}, whose length is quadratic in the dimension. Even if we cannot get strong blocking sets smaller than those in \cite{Heger:2021vr}, this construction has the great advantage of being explicit.

\begin{example}
We illustrate with a table some minimal codes that one can obtain by Corollary~\ref{thm:MDSconcat}. Since we want to maximize the rate, we choose the inner code to have the shortest length. In the binary case, we know (see \cite{alfarano2019geometric}) that the shortest minimal codes in dimensions $\leq 5$ have parameters $[3,2,2]_2$, $[6,3,3]_2$, $[9,4,4]_2$, $[13,5,5]_2$. In the ternary case, the simplex $[4,2,3]_3$ code is the shortest minimal code of dimension $2$ and we use a $[9,3,5]_3$ code with generator matrix
\[\begin{bmatrix}
1&0&0&1&2&0&0&2&2\\
0&1&0&0&0&1&2&1&2\\
0&0&1&1&1&1&1&1&1\\
\end{bmatrix}
\]
which has the shortest length in dimension $3$ (by \cite{MR3907797} or \cite[Theorem 4.10]{alfarano2019geometric}).
So we get Table~1 (in the last column we have the lower bound of Theorem \ref{thm:lowerbound} and the upper bound given by Theorem \ref{thm:upperbound}, unless otherwise specified). 
\begin{table}[!ht]
    \centering
    \begin{tabular}{c|c|c|c}
        Outer & Inner & Concatenated & Shortest length \\
        \hline
        $[3,2,2]_4$ & $[3,2,2]_2$ & $[9,4,4]_2$ & Yes\\
        $[5,3,3]_4$ & $[3,2,2]_2$ & $[15,6,6]_2$ & Yes\\
        $[3,2,2]_{16}$ & $[9,4,4]_2$ & $[27,8,8]_2$  & $21\leq n\leq 26$\\
        $[5,3,3]_8$ & $[6,3,3]_2$ & $[30,9,9]_2$ & $24\leq n\leq 40$\\
        $[3,2,2]_{32}$ & $[13,5,5]_2$ & $[39,10,10]_2$ &  $27\leq n\leq 30$ (see \cite{cohen-zemor})\\
        $[7,4,4]_8$ & $[6,3,3]_2$ & $[42,12,12]_2$ & $ 33\leq n\leq 55$\\
        $[9,5,5]_8$ & $[6,3,3]_2$ & $[54,15,15]_2$ & $ 42\leq n\leq 69$\\
        $[7,4,4]_{16}$ & $[9,4,4]_2$ & $[63,16,16]_2$  &  $45\leq n\leq 74$\\
        $[3,2,2]_{512}$ & $[30,9,9]_2$ & $[90,18,18]_2$  &  $51\leq n\leq 84$\\
        $[9,5,5]_{16}$ & $[9,4,4]_2$ & $[81,20,20]_2$  &  $57\leq n\leq 93$\\
        \hline
        $[4,2,3]_9$ & $[4,2,3]_3$ & $[16,4,9]_3$ &  $12\leq n\leq 14$ (see \cite{alfarano2019geometric})\\
        $[7,3,5]_9$ & $[4,2,3]_3$ & $[28,6,15]_3$ & $20\leq n\leq 40$ \\
        $[10,4,7]_{9}$ & $[4,2,3]_3$ & $[40,8,21]_3$ & $28\leq n\leq 56$\\
        $[7,3,5]_{27}$ & $[9,3,5]_3$ & $[63,9,26]_3$ & $32\leq n\leq 64$\\
        $[10,4,7]_{27}$ & $[9,3,5]_3$ & $[90,12,35]_3$ & $44\leq n\leq 84$\\
        $[13,5,9]_{27}$ & $[9,3,5]_3$ & $[117,15,45]_3$ & $56\leq n\leq 108$
    \end{tabular}
    \label{tab:MDS}
    \caption{MDS concatenated with shortest minimal codes}
\end{table}

We give here the generator matrix of the $[15,6,6]_2$ code obtained by the concatenation of the extended Reed-Solomon $[5,3,3]_4$ with the simplex $[3,2,2]_2$ code:
\[G:=\begin{bmatrix}
1&0&1&0&0&0&0&0&0&1&0&1&1&1&0\\
0&1&1&0&0&0&0&0&0&0&1&1&1&0&1\\
0&0&0&1&0&1&0&0&0&1&1&0&1&1&0\\
0&0&0&0&1&1&0&0&0&1&0&1&1&0&1\\
0&0&0&0&0&0&1&0&1&1&1&0&1&0&1\\
0&0&0&0&0&0&0&1&1&1&0&1&0&1&1
\end{bmatrix}\]
The length of this code meets the bound in Theorem \ref{thm:lowerbound}. In the matrix $G$, every third column is the sum of the previous two. This means that the underlying projective system is the union of $5$ disjoint lines (which are necessarily in higgledy-piggledy arrangement since the code is minimal).
By \cite[Theorem 3.12]{Heger:2021vr}, this is the minimum number of lines in higgledy-piggledy arrangement in PG$(5,2)$.
\end{example}

\begin{example}
We illustrate in Table \ref{tab:MDS2} some minimal codes that one can obtain by Corollary~\ref{Th:metaconstruction}. Since we want to maximize the rate, we choose the outer code to be the shortest code of a given dimension satisfying the Outer AB condition, in the library of codes with the best known minimum distance in {\sc Magma}.
Again, in the last column we have the lower bound of Theorem \ref{thm:lowerbound} and the upper bound given by Theorem \ref{thm:upperbound}. 

\begin{table}[!ht]
    \centering
    \begin{tabular}{c|c|c|c}
        Outer & Inner & Concatenated & Shortest length \\
        \hline
        $[ 9, 4, 5 ]_{4}$ & $[3,2,2]_2$ & $[ 27, 8, 10 ]_2$  & $21\leq n\leq 36$\\
        $[ 11, 5, 6 ]_{4}$ & $[3,2,2]_2$ & $[ 33, 10, 12 ]_2$  & $27\leq n\leq 45$\\
        $[ 15, 6, 8 ]_{4}$ & $[3,2,2]_2$ & $[ 45, 12, 16 ]_2$  & $33\leq n\leq 55$\\
        $[ 21, 7, 11 ]_{4}$ & $[3,2,2]_2$ & $[ 63, 14, 22 ]_2$  & $39\leq n\leq 65$\\
        $[ 23, 8, 12 ]_{4}$ & $[3,2,2]_2$ & $[ 69, 16, 24 ]_2$  & $45\leq n\leq 74$\\
    $ [ 16, 5, 11 ]_{9}$ & $[4,2,3]_3$ & $[ 64, 10, 33 ]_3$  & $36\leq n\leq 72$
    \end{tabular}
    \label{tab:MDS2}
    \caption{Best known linear codes satisfying Outer AB condition  concatenated with simplex codes}
\end{table}
\end{example}

\subsection{Application to $\rho$-saturating sets}\label{sec:application}

As a byproduct of our explicit  constructions of short minimal codes, we present results on $\rho$-saturating sets of small size, which have deep connections with strong blocking sets  \cite{1930-5346_2011_1_119}.

\begin{definition}
A set $\mathcal{S} \subseteq  \mathrm{PG}(k-1, q)$ is called $\rho$-\emph{saturating} if for any point $Q \in \mathrm{PG}(k-1, q) \setminus  \mathcal{S}$ there exist $\rho+1$ points  $P_1,\ldots,P_{\rho+1}\in \mathcal{S}$ such that $Q\in \langle P_1,\ldots,P_{\rho+1}\rangle$ and  $\rho$ is the smallest value with this property.  Let $s_q(k-1,\rho)$ denote the smallest size of a $\rho$-saturating set in $\mathrm{PG}(k-1, q)$.
\end{definition}

Recall that the covering radius of an $[n,n-k]_q$ code
is the least integer $R$ such that the space $\F_q^n$ is covered by
spheres of radius $R$ centered on codewords. It is easy to prove (see \cite{1930-5346_2011_1_119}) that a linear $[n,n-k]_q$ code has covering radius $R$ if every
element of $\F_q^k$ is a linear combination of $R$ columns of a generator matrix of the dual code (that is the orthogonal space with respect to the Euclidean inner product), and $R$ is the smallest value with such a  property. Thus the correspondence presented in Section \ref{sec:strongblockingsets} specializes to a correspondence between  $(R-1)$-saturating sets of size $n$ in $\mathrm{PG}(k-1,q)$ and the dual of  $[n,n-k]_q$ codes of covering radius $R$.

A connection between strong blocking sets and $\rho$-saturating sets is the following. 

\begin{theorem}[\!\!\cite{1930-5346_2011_1_119}]
Any strong blocking set in a subgeometry $\mathrm{PG}(k-1,q)$ of  $\mathrm{PG}(k-1,q^{k-1})$ is a $(k-2)$-saturating set in $\mathrm{PG}(k-1,q^{k-1})$.
\end{theorem}
Recently, improvements on the upper bound on the minimum size of a $\rho$-saturating set have been obtained in \cite{Davydov_2019,denaux2021constructing}. 

\begin{theorem}[\!\!\cite{denaux2021constructing}]
The following bound on  the minimum size of a $\rho$-saturating in $\mathrm{PG}(k-1,q^{\rho+1})$ holds 
$$\frac{\rho+1}{e}q^{k-1-\rho}\leq s_{q^{\rho+1}}(k-1,\rho)\leq (\rho+1)\left((\rho+2)\frac{q^{k-1-\rho}}{2}+\rho \frac{q^{k-1-\rho}-1}{q-1}\right).$$
\end{theorem}

\noindent In particular, for $\rho=k-2$, 
\begin{equation}\label{Eq:saturating2}
    s_{q^{k-1}}(k-1,k-2)\leq q\binom{k}{2}+(k-1)(k-2).
\end{equation}

The approach described in Section \ref{subsection:explicit constructions} yields an explicit construction of small saturating sets. We use the same notation:
\begin{itemize}
    \item $q_0$ is a prime power;
    \item $N_{2,n}:= (q_0^{2(n+1)}-q_0^{2n}-1)$;
    \item $\lambda_{2,n}:=\begin{cases}
(q_0^{n}-1)^2,& \textrm{if } 2\mid n;\\
(q_0^{n+1}-1)(q_0^{n-1}-1),& \textrm{if } 2\nmid n;\\
\end{cases}$
\item $K_n:=4\left\lfloor\frac{N_{2,n}-1}{q_0}\right\rfloor-4\lambda_{2,n}+4$.
\end{itemize}

\begin{theorem}
In $\mathrm{PG}(K_n-1 ,q_0^{K_n-1})$ there exists an explicit construction of $(K_n-2)$-saturating set of size   at most  
\begin{equation}\label{Eq:saturating}
\left(\frac{q_0(q_0^4-1)(q_0+1)}{4(q_0^2-q_0-1)}\right)\cdot K_n\leq q_0\cdot \frac{q_0^3+2q_0^2+3q_0+5}{4} \cdot K_n\simeq \frac{q_0^4}{4}K_n.\end{equation}
\end{theorem}

Note that although Bound \eqref{Eq:saturating} is worse than the estimate  in \cite[Corollary 6.4]{Heger:2021vr}, it provides the best explicit construction of small saturating sets for specific dimensions and it improves   \eqref{Eq:saturating2}.

\section*{Acknowledgments}
This work was started during the first author's stay at the LAGA at University Paris 8 as a Visiting Professor in 2021. The research of D. Bartoli was supported by  the Italian National Group for Algebraic and Geometric Structures and their Applications (GNSAGA - INdAM). The second author was partially supported by the ANR-21-CE39-0009 - BARRACUDA (French \emph{Agence Nationale de la Recherche}). The authors thank the anonymous reviewers whose detailed comments helped to improve the presentation of the paper.

\bibliographystyle{abbrv}
\bibliography{references.bib}

\begin{thebibliography}{10}

\bibitem{alfarano2019geometric}
G.~N. Alfarano, M.~Borello, and A.~Neri.
\newblock A geometric characterization of minimal codes and their asymptotic
  performance.
\newblock {\em Advances in Mathematics of Communications}, 16(1):115, 2022.

\bibitem{alfarano2020three}
G.~N. Alfarano, M.~Borello, A.~Neri, and A.~Ravagnani.
\newblock Three combinatorial perspectives on minimal codes.
\newblock {\em to appear in SIAM J. Discrete Math.}, 2021.

\bibitem{alfaranolinear}
G.~N. Alfarano, M.~Borello, A.~Neri, and A.~Ravagnani.
\newblock Linear cutting blocking sets and minimal codes in the rank metric.
\newblock {\em Journal of Combinatorial Theory, Series A}, 2022.

\bibitem{ashikhmin1998minimal}
A.~Ashikhmin and A.~Barg.
\newblock Minimal vectors in linear codes.
\newblock {\em IEEE Transactions on Information Theory}, 44(5):2010--2017,
  1998.

\bibitem{Bartoli:2020tk}
D.~Bartoli, A.~Cossidente, G.~Marino, and F.~Pavese.
\newblock On cutting blocking sets and their codes.
\newblock In {\em Forum Mathematicum}. De Gruyter, 2022.

\bibitem{Bartoli_2020}
D.~Bartoli, G.~Kiss, S.~Marcugini, and F.~Pambianco.
\newblock Resolving sets for higher dimensional projective spaces.
\newblock {\em Finite Fields Appl.}, 67:101723, 14, 2020.

\bibitem{beutelspacher1983}
A.~Beutelspacher.
\newblock On {B}aer subspaces of finite projective spaces.
\newblock {\em Math. Z.}, 184(3):301--319, 1983.

\bibitem{MR3907797}
A.~Bishnoi, S.~Mattheus, and J.~Schillewaert.
\newblock Minimal multiple blocking sets.
\newblock {\em Electron. J. Combin.}, 25(4):Paper No. 4.66, 14, 2018.

\bibitem{bs}
A.~Blokhuis, P.~Sziklai, and T.~Szőnyi.
\newblock Blocking sets in projective spaces.
\newblock {\em {\rm In} Current Research Topics in Galois Geometry {\rm (De
  Beule, J. and Storme, L., eds)}, Nova Academic}, pages 61--84, 2011.

\bibitem{bonini2020minimal}
M.~Bonini and M.~Borello.
\newblock Minimal linear codes arising from blocking sets.
\newblock {\em J. Algebraic Combin.}, 53(2):327--341, 2021.

\bibitem{Calderbank_Kantor}
R.~Calderbank and W.~M. Kantor.
\newblock The geometry of two-weight codes.
\newblock {\em Bull. London Math. Soc.}, 18(2):97--122, 1986.

\bibitem{Chabanne_2014}
H.~Chabanne, G.~Cohen, and A.~Patey.
\newblock Towards secure two-party computation from the wire-tap channel.
\newblock In {\em Information Security and Cryptology -- {ICISC} 2013}, pages
  34--46. Springer International Publishing, 2014.

\bibitem{MR3644879}
G.~Cohen and S.~Mesnager.
\newblock Variations on minimal linear codes.
\newblock In {\em Coding theory and applications}, volume~3 of {\em CIM Ser.
  Math. Sci.}, pages 125--131. Springer, Cham, 2015.

\bibitem{Cohen_2013}
G.~D. Cohen, S.~Mesnager, and A.~Patey.
\newblock On minimal and quasi-minimal linear codes.
\newblock In {\em Cryptography and Coding}, pages 85--98. Springer Berlin
  Heidelberg, 2013.

\bibitem{cohen-zemor}
G.~D. Cohen and G.~Z\'emor.
\newblock Intersecting codes and independent families.
\newblock {\em IEEE Transactions on Information Theory}, 40(6):1872--1881,
  1994.

\bibitem{1930-5346_2011_1_119}
A.~A. Davydov, M.~Giulietti, S.~Marcugini, and F.~Pambianco.
\newblock Linear nonbinary covering codes and saturating sets in projective
  spaces.
\newblock {\em Adv. Math. Commun.}, 5(1):119--147, 2011.

\bibitem{Davydov_2019}
A.~A. Davydov, S.~Marcugini, and F.~Pambianco.
\newblock New covering codes of radius {$R$}, codimension {$tR$} and
  {$tR+\frac{R}{2}$}, and saturating sets in projective spaces.
\newblock {\em Des. Codes Cryptogr.}, 87(12):2771--2792, 2019.

\bibitem{denaux2021constructing}
L.~Denaux.
\newblock Constructing saturating sets in projective spaces using
  subgeometries.
\newblock {\em Designs, Codes and Cryptography}, pages 1--32, 2021.

\bibitem{zbMATH06340223}
S.~L. {Fancsali} and P.~{Sziklai}.
\newblock {Lines in higgledy-piggledy arrangement}.
\newblock {\em {Electron. J. Comb.}}, 21(2):research paper p2.56, 15, 2014.

\bibitem{MR1423052}
A.~Garcia and H.~Stichtenoth.
\newblock On the asymptotic behaviour of some towers of function fields over
  finite fields.
\newblock {\em J. Number Theory}, 61(2):248--273, 1996.

\bibitem{Heger:2021vr}
T.~H{\'e}ger and Z.~L. Nagy.
\newblock Short minimal codes and covering codes via strong blocking sets in
  projective spaces.
\newblock {\em IEEE Transactions on Information Theory}, 2021.

\bibitem{MR551274}
T.~Y. Hwang.
\newblock Decoding linear block codes for minimizing word error rate.
\newblock {\em IEEE Trans. Inform. Theory}, 25(6):733--737, 1979.

\bibitem{zbMATH03803619}
Y.~{Ihara}.
\newblock {Some remarks on the number of rational points of algebraic curves
  over finite fields}.
\newblock {\em {J. Fac. Sci., Univ. Tokyo, Sect. I A}}, 28:721--724, 1981.

\bibitem{lidl_niederreiter_1994}
R.~Lidl and H.~Niederreiter.
\newblock {\em Introduction to Finite Fields and their Applications}.
\newblock Cambridge University Press, 2 edition, 1994.

\bibitem{Lu_2021}
W.~Lu, X.~Wu, and X.~Cao.
\newblock The parameters of minimal linear codes.
\newblock {\em Finite Fields Appl.}, 71:101799, 11, 2021.

\bibitem{Massey1993}
J.~L. Massey.
\newblock Minimal codewords and secret sharing.
\newblock In {\em Proc. 6th Joint Swedish-Russian Int. Workshop on Info.
  Theory}, pages 276--279, 1993.

\bibitem{Massey1995}
J.~L. Massey.
\newblock Some applications of coding theory in cryptography.
\newblock In {\em Codes and Cyphers: Cryptography and Coding IV}, pages 33--47,
  1995.

\bibitem{8915756}
S.~Mesnager and A.~Sınak.
\newblock Several classes of minimal linear codes with few weights from weakly
  regular plateaued functions.
\newblock {\em IEEE Transactions on Information Theory}, 66(4):2296--2310,
  2020.

\bibitem{SHI202160}
M.~Shi and X.~Li.
\newblock Two classes of optimal {$p$}-ary few-weight codes from down-sets.
\newblock {\em Discrete Appl. Math.}, 290:60--67, 2021.

\bibitem{SHI2020111840}
Z.~Shi and F.-W. Fu.
\newblock Several families of $q$-ary minimal linear codes with
  $w_{min}/w_{max}\leq (q-1)/q$.
\newblock {\em Discrete Mathematics}, 343(6):111840, 2020.

\bibitem{Full_Characterization}
C.~Tang, Y.~Qiu, Q.~Liao, and Z.~Zhou.
\newblock Full characterization of minimal linear codes as cutting blocking
  sets.
\newblock {\em IEEE Trans. Inform. Theory}, 67(6):3690--3700, 2021.

\bibitem{tsfasman1991algebraic}
M.~A. Tsfasman and S.~G. Vl{\u{a}}du{\c{t}}.
\newblock {\em Algebraic-geometric codes}, volume~58 of {\em Mathematics and
  its Applications (Soviet Series)}.
\newblock Kluwer Academic Publishers Group, Dordrecht, 1991.
\newblock Translated from the Russian by the authors.

\bibitem{zbMATH03918262}
M.~A. {Tsfasman}, S.~G. {Vl\u{a}dut}, and T.~{Zink}.
\newblock {Modular curves, Shimura curves, and Goppa codes, better than
  Varshamov-Gilbert bound}.
\newblock {\em {Math. Nachr.}}, 109:21--28, 1982.

\end{thebibliography}

\end{document}